\documentclass[onefignum,onetabnum]{siamart190516}

\usepackage{amsmath}
\usepackage{amsfonts}
\usepackage{dsfont}
\usepackage{amssymb}
\usepackage{bbold}
\usepackage{mathtools}
\usepackage{enumerate}
\usepackage{color}
\usepackage[normalem]{ulem}

\newcommand{\nicolasg}[1]{#1}

\allowdisplaybreaks
\newcommand*\diff{\mathop{}\!\mathrm{d}}

\newcommand{\R}{\mathds{R}}
\newcommand{\A}{\mathds{A}}
\newcommand{\1}{\mathds{1}}
\newcommand{\N}{\mathds{N}}
\newcommand{\M}{\mathcal{M}}
\newcommand{\W}{\mathds{W}}
\newcommand{\spc}{\;\;}

\newcommand{\norm}[1]{\|#1\|}
\let\phi\varphi

\newcommand{\X}{\mathds{X}}

\newcommand{\Z}{\mathds{Z}}

\newcommand{\sB}{\mathcal{B}}

\newcommand{\sM}{\mathcal{M}}
\newcommand{\sX}{\mathcal{X}}
\newcommand{\sY}{\mathcal{Y}}

\newcommand{\sF}{\mathcal{F}}

\newcommand{\Ceil}[1]{\lceil #1 \rceil}

\newcommand{\inv}{\mathrm{inv}}%

\usepackage{lipsum}
\usepackage{graphicx}
\usepackage{epstopdf}
\usepackage{algorithmic}
\ifpdf
  \DeclareGraphicsExtensions{.eps,.pdf,.png,.jpg}
\else
  \DeclareGraphicsExtensions{.eps}
\fi

\newsiamremark{remark}{Remark}
\newsiamremark{hypothesis}{Hypothesis}
\crefname{hypothesis}{Hypothesis}{Hypotheses}
\newsiamthm{claim}{Claim}

\headers{Ergodicity of control systems under information constraints}{N.~Garcia, C.~Kawan, and S.~Y\"{u}ksel}

\title{ERGODICITY CONDITIONS FOR CONTROLLED STOCHASTIC NONLINEAR SYSTEMS UNDER INFORMATION CONSTRAINTS: A Volume Growth Approach\thanks{A brief version of this paper has been scheduled for presentation at the 24th International Symposium on Mathematical Theory of Networks and
Systems (MTNS 2021). \funding{This work was funded by the Natural Sciences and Engineering Research Council of Canada.}}}

\author{Nicol\'{a}s Garcia\thanks{Is with the Department of Mathematics and Statistics at Queen's University, Kingston ON, Canada. Email: 11ng31@queensu.ca}
\and Christoph Kawan\thanks{Is with the Institute of Informatics at the LMU Munich, Germany. He is supported by the German Research Foundation (DFG) through the grant ZA 873/4-1. Email: christoph.kawan@lmu.de}
\and Serdar Y\"{u}ksel\thanks{Is with the Department of Mathematics and Statistics at Queen's University, Kingston ON, Canada. Email: yuksel@queensu.ca}
}

\usepackage{amsopn}

\begin{document}

\maketitle

\begin{abstract}
Consider a stochastic nonlinear system controlled over a possibly noisy communication channel. An important problem is to characterize the largest class of channels which admit coding and control policies so that the closed-loop system is stochastically stable. In this paper we consider the stability notion of (asymptotic) ergodicity. We prove lower bounds on the channel capacity necessary to achieve the stability criterion. Under mild technical assumptions, we obtain that the necessary channel capacity is lower bounded by the log-determinant of the linearization, double-averaged over the state and noise space. We prove this bound by introducing a modified version of invariance entropy, and utilizing the almost sure convergence of sample paths guaranteed by the pointwise ergodic theorem. Our results generalize those for linear systems, and are in some cases more refined than those obtained for nonlinear systems via information-theoretic methods.
\end{abstract}



\section{Introduction}

In this paper, we consider a stochastic nonlinear system controlled over a possibly noisy communication channel. We consider the problem of determining necessary conditions on channel capacity required for the existence of coding and control policies so that the closed-loop system is stochastically stable. The stability criterion considered is asymptotic ergodicity, by which we mean the existence of an asymptotically mean stationary measure which is also ergodic. Our analysis considers systems of the form%
\begin{equation*}
	x_{t+1} = f(x_t,w_t) + u_t%
\end{equation*}
where $x_t$ and $u_t$ take values in $\R^N$ and $w_t$ takes values in an abstract probability space. The variables $x_t,w_t$ and $u_t$ represent the state, noise, and control action at time $t$, respectively. The noise is modeled in an i.i.d.~fashion and the initial state $x_0$ is considered random and independent of the noise variables.%

In the case of a deterministic system, the notion of invariance entropy has been used to study a related problem, namely stabilization in the sense of set-invariance \cite{colonius2009invariance}. The invariance entropy of a compact subset $Q$ of the state space is defined as%
\begin{equation*}
  h_{\inv}(Q) \coloneqq \lim_{\tau \to \infty}\frac{1}{\tau}\log r_{\inv}(\tau,Q),%
\end{equation*}
where $r_{\inv}(\tau,Q)$ is the minimum number of control inputs required to make $Q$ invariant on the time interval $[0,\tau]$ for arbitrary initial states in $Q$. The invariance entropy measures the smallest average rate of information that must be transmitted to a controller to render $Q$ invariant.%

The motivation for the above definition arises by observing that with $n$ bits of information available at the controller side, at most $2^n$ different states can be distinguished, and therefore at most $2^n$ different control inputs can be generated. 

In the case of stochastic systems, this reasoning does not apply directly: (i) Asking for a compact subset of the state space to be invariant is too restrictive to be a useful notion of stability. For example, if the system is subject to unbounded noise, the state process may leave a given compact set regardless of the control policy. Therefore, we consider here instead notions of stochastic stability such as ergodicity and asymptotic mean stationarity (AMS). (ii) If the channel is noisy, the informational content of received codewords cannot be measured by the number of distinct possible receiver outputs. As an extreme case, consider a channel where the channel inputs and outputs are independent, and hence the (information-theoretic) channel capacity is zero. In this case, no reliable information can be transmitted across the channel, and thus, the methodology presented above for noise-free models by means of a direct application of invariance entropy is no longer applicable. On the other hand, the information-theoretic approach for this problem does not allow one to develop a geometric analytical refinement afforded by a stochastic volume growth approach; and one of our main contributions in this paper is to develop a framework, alternative to methods building on directed mutual information \cite{yuksel2016stationary}, to approach the study of nonlinear systems controlled over noisy channels.%

In \cite{kawan2019invariance}, the notion of invariance entropy was generalized for use in the stability analysis of discrete-time stochastic systems controlled over finite-capacity channels. The introduced quantity, called \emph{stabilization entropy}, is inspired by both invariance entropy and measure-theoretic entropy of dynamical systems, in particular by a characterization of the latter due to Katok \cite{katok1980lyapunov} and a generalization thereof developed in Ren et al.~\cite{ren2011topological}.%

In the paper at hand, we provide an operationally and mathematically significant refinement, where our stability criterion is stochastic in nature, but deterministic in its sample path limits, as we will make precise further below. Our stronger notion of stability guarantees the almost sure convergence of sample paths which asymptotically visit each subset of the state space at a frequency given by the AMS measure of the subset. We further generalize the notion of stabilization entropy by considering a finite collection of subsets rather than one single subset of the state space, and prove stronger results, using the pointwise ergodic theorem.%

The paper is organized as follows. Section \ref{sec:lr} provides a brief literature review and presents our contributions. Some fundamental definitions and technical tools are introduced in Section \ref{sec:perlims}. The main results are presented and discussed in Section \ref{sec:boc}, while their proofs are given in Section \ref{sec:pfs}. Some definitions and auxiliary results are outlined in the appendix.%

\section{Literature review and contributions}\label{sec:lr}

The problem of determining necessary and sufficient conditions for stochastic stability of Markov chains, in the form of the existence of a stationary measure and positive Harris recurrence, has been studied using Lyapunov methods and we refer the reader to \cite{meyn2012markov} for a comprehensive treatment. To implement stabilizing control policies however, full feedback is often required (or in case of partially observed models, restrictive invertibility conditions related to observability are needed), a condition which is too restrictive in many modern application areas. For example, the controller may have access only to an estimate of the state encoded in $n$ bits at each time step, in which case the typically uncountable state space must be quantized using a finite ($2^n$) number of symbols. As such, the assumption that the controller has full state access, at arbitrary levels of precision, is no longer valid. In particular, this is the case in networked systems, where communication resources have to be distributed among many agents, and in underwater applications, where communication is naturally constrained due to the physical properties of the environment. The emergence of such problems has motivated the study of control problems subject to information constraints, and the development of the general theory of information-based control.%

In the case of linear systems, explicit formulas have been obtained for the smallest channel capacity above which stabilization is possible. Under certain stability notions, the capacity of the channel must not be smaller than the logarithm of the unstable open-loop determinant. The earliest contributions can be found in \cite{wong1999systems,baillieul1999feedback}. These formulas, known as data-rate theorems, were further generalized in \cite{hespanha2002towards}, \cite{nair2004stabilizability} and \cite{tatikonda2004control}. For a more complete discussion of related results, see \cite{andrievsky2010control,franceschetti2014elements,matveev2009estimation,nair2007feedback}.%

For nonlinear systems, most of the results in the literature have been obtained for deterministic systems controlled over noiseless channels. To this end, the notion of topological feedback entropy was introduced in \cite{nair2004topological} for the study of discrete-time systems. A related result, by the same authors, is a characterization of the smallest data rate required for stabilization to an equilibrium point as the log-sum of the unstable eigenvalues of the linearization. For the case of continuous-time systems, the notion of invariance entropy was introduced in \cite{colonius2009invariance}. Both topological feedback entropy and invariance entropy capture the smallest average rate of information required to keep the state inside a compact set. When adapted to the discrete-time setting, the two notions are equivalent, as was shown in \cite{colonius2013note}. An extensive review of these concepts is provided in \cite{kawan2013invariance}. A recent related development was the introduction of metric invariance entropy in \cite{colonius2018metric}, a notion based on conditionally invariant measures.%

Other studies on control of nonlinear systems over communication channels have focused on constructive schemes (and not on converse theorems), primarily for noise-free systems and channels, cf.~\cite{liberzon2005stabilization,de2004stabilizability,li2004data}. Recently, necessary conditions in the form of lower bounds on the channel capacity for a certain class of stochastic nonlinear systems over both noiseless and noisy channels were established in \cite{kawan2019invariance} and \cite{yuksel2016stationary}, where the stability notion considered in the first paper is AMS, and the notions considered in the second paper are AMS, ergodicity, and positive Harris recurrence. In \cite{kawan2019invariance}, the notion of stabilization entropy is used, while \cite{yuksel2016stationary} relies on information-theoretic techniques, where the different approaches arrive at complementary results.%

For a class of nonlinear systems controlled over noiseless channels \cite{yuksel2016stationary}, and for linear systems over Gaussian, discrete noiseless, erasure and discrete noisy channels \cite{yuksel2013stochastic}, as well as \cite{YukTAC2010,YukselAMSITArxiv,YukMeynTAC2010} establish the ergodicity property under information constraints. In particular, \cite[Thm.~4.2]{YukselAMSITArxiv} shows that for a linear system with a diagonalizable matrix $A$, controlled over a DMC, the AMS and ergodicity properties can be achieved whenever the channel capacity exceeds the log-sum of the unstable eigenvalues. Hence, in this case the lower bounds following from the results in this paper match with the upper bound. For nonlinear systems of the form%
\begin{equation*}
  x_{t+1} = f(x_t,u_t) + w_t%
\end{equation*}
with $f(\cdot,u):\R^N \rightarrow \R^N$ invertible and $C^1$ for every $u$ and $\{w_t\}$ an i.i.d.~sequence of zero-mean Gaussian variables, it is shown in \cite[Thm.~5.1]{yuksel2016stationary} that ergodicity (and thus AMS) can be achieved over a over a discrete noiseless channel under the following assumption: There exist a function $\kappa:\R^N \rightarrow \R^M$ with $\kappa(0)=0$ and a constant $a>0$ such that $|f(x,\kappa(z))|_{\infty} \leq a |x - z|_{\infty}$ for all $x,z\in\R^N$. In this case, the minimal required channel capacity $C_0$ satisfies $C_0 \leq N \log(a) + 1$. Therefore, the goal of ergodicity is attainable even for systems with additive unbounded noise. Though not directly related, further relevant papers on the general subject of nonlinear control under information constraints include \cite{Mehta1,Mehta2,Martins2,diwadkar2012limitations,vaidya2012stabilization,zang2003nonlinear}.%

It is important to note that for linear systems, any local dynamical or control-theoretic property is also a global property. As such, the problems of local stabilization (stabilization to a point), semi-global stabilization (set-invariance) and global stabilization (stochastic stability) can all be handled with similar methods, leading to the aforementioned data-rate theorem in each case. For nonlinear systems, however, the three stability problems are fundamentally different and require distinct approaches. For example, linearization techniques work well for local problems, for semi-global problems only under specific assumptions, and almost never for global problems. In addition, the presence of (possibly unbounded and additive) noise requires an approach fundamentally different from the machinery utilized for local stabilization problems.%

{\bf Contributions.} In this paper, we study the problem of stochastic stabilization of a nonlinear stochastic system controlled over a finite-capacity communication channel, with the stability criterion being the (asymptotic) ergodicity of the process. As a primary contribution, we develop a stochastic volume growth technique tailored to ergodicity properties, which is in contrast with the information-theoretic methods studied earlier, and establish refined and more general results on information transmission requirements for making the controlled stochastic nonlinear system ergodic. In particular, compared with \cite{yuksel2016stationary}, we allow arbitrary coding and control policies and do not impose an entropy growth condition apriori. Our results generalize the linear setups considered extensively in the literature.%

\section{Preliminaries}\label{sec:perlims}

\subsection{Notation}

Throughout the paper, $\N$ denotes the strictly positive integers, $\Z_+$ denotes $\N \cup \{0\}$ and $\R_{>0}$ the strictly positive real numbers. We write $[a;b]$ for a discrete interval, i.e., $[a;b] = \{a,a+1,\ldots,b\}$ for any $a \leq b$ in $\Z$. The notation $\sB(\mathds{X})$ is used for the Borel $\sigma$-algebra of a Polish space $\mathds{X}$. Furthermore, $\Sigma$ denotes the space of sequences in a Polish space $\X$, i.e., $\Sigma = \X^{\Z_+}$, and $\sB(\Sigma)$ the Borel $\sigma$-algebra of $\Sigma$, which is generated by cylinder sets. If $x \in \X^{\Z_+}$, we write $x_{[0,t]} = (x_0,x_1,\ldots,x_t)$ for any $t \in \Z_+$. By $m$ we denote the Lebesgue measure on $\R^N$ for any $N\in\N$. All logarithms are taken to the base $2$.%

\subsection{Stochastic stability and ergodic properties}

In this section, we provide some basic definitions, and characterize the stability criterion considered in this paper: asymptotic mean stationarity with the associated AMS measure resulting in an ergodic state process. We refer the reader to \cite{GrayKieffer,gray2009probability,GrayInfo} for a detailed study of the concepts presented in this section.%

First, recall some basic facts from ergodic theory: A measurable map $T:\Omega \rightarrow \Omega$ on a probability space $(\Omega,\sF,P)$ is called \emph{measure-preserving} if $P(T^{-1}(A)) = P(A)$ for all $A \in \sF$. An event $A \in \sF$ is \emph{$T$-invariant} if $A = T^{-1}(A)$ (up to a set of measure zero). We denote by $\sF_{\inv(T)}$ the set of all $T$-invariant measurable sets, which is a $\sigma$-algebra. A measure-preserving map $T$ is called \emph{ergodic} if $P(A) \in \{0,1\}$ for all $A \in \sF_{\inv(T)}$. Note that ergodicity is a property of a system $(\Omega,\sF,P,T)$, but sometimes we also say that ``$T$ is ergodic'', or occasionally ``$P$ is ergodic'', when the other components of the system are clear from the context.%

A fundamental result in ergodic theory is the following pointwise ergodic theorem.

\begin{theorem}
Let $(\Omega,\sF,P)$ be a probability space and $T:\Omega \rightarrow \Omega$ a measure-preserving map. Then for any $f \in L^1(\Omega,\sF,P)$ we have%
\begin{equation*}
  \frac{1}{N}\sum_{k=0}^{N-1} f \circ T^k \xrightarrow[N \rightarrow \infty]{a.s.} \phi%
\end{equation*}
for some $\phi \in L^1(\Omega,\sF_{\inv(T)},P|_{\sF_{\inv(T)}})$ satisfying $\int \phi \diff P = \int f \diff P$. If, in addition, $T$ is ergodic, then $\phi$ is almost everywhere constant and thus%
\begin{equation*}
  \frac{1}{N}\sum_{k=0}^{N-1} f \circ T^k \xrightarrow[N \rightarrow \infty]{a.s.} \int f \diff P.%
\end{equation*}
\end{theorem}

In the following, we fix a Polish space $\X$ and the associated sequence space $\Sigma = \X^{\Z_+}$. The shift map on $\Sigma$ is defined by%
\begin{equation*}
  \theta:\Sigma \rightarrow \Sigma,\quad (\theta x)_t :\equiv x_{t+1}\quad \mbox{for all\ } x = (x_t)_{t\in\Z_+} \in \Sigma.%
\end{equation*}
A measure $\mu$ on $(\Sigma,\sB(\Sigma))$ is called \emph{stationary} if $\mu(\theta^{-1}(B)) = \mu(B)$ for all $B \in \sB(\Sigma)$, i.e., if $(\Sigma,\sB(\Sigma),\mu,\theta)$ is a measure-preserving system.%

A stochastic process $x = (x_t)_{t \in \Z_+}$ taking values in $\X$ (with underlying probability space $(\Omega,\sF,P)$) is called%
\begin{itemize}
\item \emph{stationary} if its process measure is stationary, i.e., $P(\{\omega \in \Omega : x(\omega) \in B\}) = P(\{\omega \in \Omega : (\theta x)(\omega) \in B\})$ for all $B \in \sB(\Sigma)$.%
\item \emph{asymptotically mean stationary (AMS)} if there exists a probability measure $\tilde{Q}$ on $(\Sigma,\sB(\Sigma))$ such that%
\begin{equation*}
	\lim\limits_{T\rightarrow \infty} \frac{1}{T}\sum_{t=0}^{T-1} P(\theta^{-t}(B)) = \tilde{Q}(B) \quad \mbox{for all\ } B \in \sB(\Sigma).%
\end{equation*}
It can easily be shown that the measure $\tilde{Q}$ is stationary. We can also obtain a measure $Q$ on $(\X,\sB(\X))$ by projecting $\tilde{Q}$ down to any of its coordinates. It follows that for any $B \in \sB(\X)$%
\begin{equation*}
  \lim\limits_{T\rightarrow \infty} \frac{1}{T}\sum_{t=0}^{T-1} P(x_t \in B) = Q(B).%
\end{equation*}
\end{itemize}

Let $\mu$ denote the process measure on $(\Sigma,\sB(\Sigma))$ and suppose that the system $(\Sigma,\sB(\Sigma),\mu,\theta)$ is ergodic. Observe that for a set $B \in \sB(\X)$, by the pointwise ergodic theorem, we have%
\begin{equation*}
  \mu\Bigl(\Bigl\{x \in\Sigma : \lim_{T \to \infty}\frac{1}{T}\sum_{t=0}^{T-1}\1_B(x_t) = \int \1_B(x_0)\diff \mu(x)\Bigr\}\Bigr) = 1%
\end{equation*} 
which we can (using the notation $\mu$ also for the projection of $\mu$ to $\X$) rewrite as%
\begin{align}\label{spce}
  \mu\Bigl(\Bigl\{x \in\Sigma : \lim_{T \to \infty}\frac{1}{T}\sum_{t=0}^{T-1}\1_B(x_t) = \mu(B)\Bigr\}\Bigr) = 1.%
\end{align}
Thus, if the stochastic process is ergodic, then the set of sample paths which visit a Borel set $B$ with frequency $\mu(B)$ is of full measure. This is a strong notion of stability, and is a key ingredient in the proofs of the theorems in this paper.%

However, we can relax ergodicity of the process measure somewhat and it turns out that (\ref{spce}) holds for a larger class of processes:%

\begin{definition}
Consider a stochastic process which is AMS with asymptotic mean $Q$. If $Q$ is ergodic, we call the process AMS ergodic.
\end{definition}

\begin{proposition}
An AMS ergodic process satisfies an equation similar to \eqref{spce}. Namely, for any $B \in \sB(\X)$ it holds that%
\begin{equation}\label{def_eq}
  \mu\Bigl(\Bigl\{x \in\Sigma : \lim_{T \to \infty}\frac{1}{T}\sum_{t=0}^{T-1}\1_B(x_t) = Q(B)\Bigr\}\Bigr) = 1.%
\end{equation}
\end{proposition}

\begin{proof}
Let us fix a $B \in \sB(\X)$. By stationarity, we can project $Q$ to the space $\X$. By a slight abuse of notation, we also denote the projected measure by $Q$. We define%
\begin{equation*}
  F \coloneqq \Bigl\{ x \in \Sigma : \lim_{T \to \infty}\frac{1}{T}\sum_{t=0}^{T-1}\1_B(x_t) = Q(B) \Bigr\}.%
\end{equation*}
From the ergodicity assumption on $Q$, it follows that $Q(F) = 1$. Also, $F$ is invariant under $\theta$ from which we obtain that $\mu(F) = 1$ (see \cite[Lem.~7.5 and Eq.~(7.22)]{gray2009probability}). See also \cite[Thm.~7.6]{gray2009probability}.
\end{proof}

\section{Information transmission rate conditions for ergodicity}\label{sec:boc}

We now state the main contributions of this paper. Proofs can be found in the next section. Consider the system%
\begin{align}	\label{mainsys}
  x_{t+1} = f(x_t,w_t) + u_t%
\end{align}
where $x_t$ and $u_t$ are $\R^N$-valued for some $N\in\N$ and $w_t$ takes values in a standard probability space $\W$. For a fixed $w \in \W$, let us denote the map $x \mapsto f(x,w)$ by $f_w$. Suppose also that the following holds:%
\begin{enumerate}
\item[(A1)] The map $f:\R^N \times \mathds{W} \rightarrow \R^N$ is Borel measurable.%
\item[(A2)] The noise process $(w_t)_{t\in\Z_+}$ is i.i.d. By abuse of notation, $\nu$ denotes both the law of any individual $w_t$, as well as the process measure.%
\item[(A3)] The map $f_w:\R^N \rightarrow \R^N$ is $C^1$ and injective for any $w \in \mathds{W}$.%
\item[(A4)] The initial state $x_0$ is random and independent of the noise process. We write $\pi_0$ for the associated probability measure.%
\item[(A5)] The measure $\pi_0$ is absolutely continuous with respect to the $N$-dimensional Lebesgue measure $m$, and its density (which exists by the Radon-Nikodym theorem) is bounded.%
\item[(A6)] There is a constant $c>0$ with $|\det Df_w(x)| > c$ for all $x \in \R^N$ and $w \in \mathds{W}$.%
\end{enumerate}

We write $(\Omega,\sF,P)$ for the probability space on which both $x_0$ and $w_t$ are modeled.%

We assume that the system is controlled over a possibly noisy communication channel as depicted in Fig.~\ref{LLL1}. The channel has a finite input alphabet $\sM$ and a finite output alphabet $\sM'$. The channel input $q_t$ at time $t$ is generated by a function $\gamma^e_t$ so that $q_t = \gamma^e_t(x_{[0,t]},q'_{[0,t-1]})$. The channel maps $q_t$ to $q'_t$ in a stochastic fashion so that $P(q_t' \in \cdot |q_t,q_{[0,t-1]},q'_{[0,t-1]}) = P(q'_t \in \cdot |q_t)$ is a conditional probability measure on $\sM'$ for all $t \in \Z_+$, for every realization $q_t,q_{[0,t-1]},q'_{[0,t-1]}$. The controller, upon receiving the information from the channel, generates its decision at time $t$, also causally: $u_t = \gamma_t^c(q'_{[0,t]})$. Any coding and control policy of this kind is called \emph{causal}. If the channel is noiseless, we have $\sM = \sM'$ and the channel capacity reduces to $C = \log|\sM|$. If the channel is noisy and memoryless, feedback does not increase its capacity, see Section \ref{infoBackground}.

\begin{figure}[h]
\begin{center}
\includegraphics[height=2.5cm,width=8.5cm]{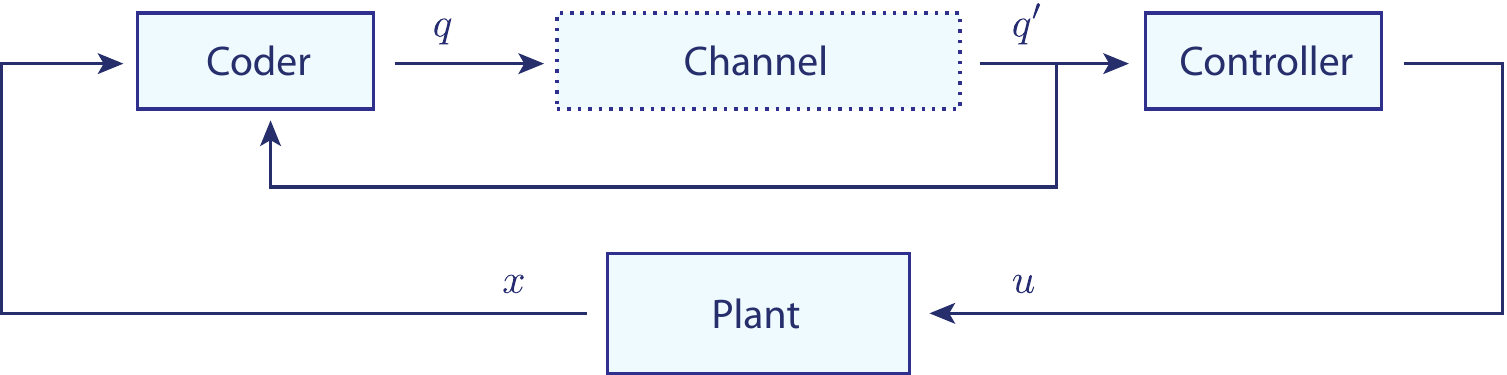}
\caption{Control of a system over a noisy channel with feedback\label{LLL1}}
\end{center}
\end{figure}

\begin{theorem}\label{thm;NewTheorem1}
Consider system \eqref{mainsys} satisfying assumptions (A1)--(A6). Suppose the system is controlled over a discrete noiseless channel of capacity $C$ and a coding and control policy achieves that the state process is AMS ergodic with asymptotic mean $Q$. Then the capacity must satisfy%
\begin{equation*}%
  \int \int \log|\det Df_w(x)|\, \diff Q(x)\diff \nu(w) \leq C.%
\end{equation*}
\end{theorem}

Our second main theorem relaxes the condition of the channel being noiseless. On the other hand, the class of nonlinear systems considered is more restrictive.%

\begin{theorem}\label{thm;NewTheorem2}
Consider the scalar system%
\begin{equation*}
  x_{t+1} = f(x_t,w_t) + u_t%
\end{equation*}
satisfying assumptions (A1)--(A5). Additionally, suppose that the following holds:%
\begin{enumerate}[(i)]
\item $|f_w'(x)| \geq 1$ for every $x \in \R$.
\item The support of $\pi_0$ is a compact interval $K \subseteq \R$.
\item The essential infimum and supremum of the density of $\pi_0$, denoted by $\rho_{\min}$ and $\rho_{\max}$, respectively, satisfy $0 < \rho_{\min} \leq \rho_{\max} < \infty$.%
\end{enumerate}
Suppose that the system is controlled over a discrete memoryless channel with feedback of capacity $C$ and a causal coding and control policy results in the state process being AMS ergodic with asymptotic mean $Q$. Then the channel capacity must satisfy%
\begin{equation}\label{Theorem2Ineq}
  \int \int \log|f_w'(x)| \diff Q(x) \diff \nu(w) \leq C.%
\end{equation}
\end{theorem}

The first theorem above is a counterpart to \cite[Thm.~5.1]{kawan2019invariance}, where it was shown for systems of the form $x_{t+1} = f(x_t) + w_t + u_t$, without the ergodicity assumption on the AMS measure, that for any Borel set $B$ of finite Lebesgue measure%
\begin{equation*}
  Q(B)\inf\limits_{x\in B}\log|\det Df(x)| \leq C%
\end{equation*}
must be satisfied. The second theorem above is a counterpart to \cite[Thm.~7.1]{kawan2019invariance} without the ergodicity assumption on the AMS measure.%

To prove \cref{thm;NewTheorem1} and \cref{thm;NewTheorem2}, the stabilization entropy introduced in \cite{kawan2019invariance} must be generalized and a technical lemma proven. This is carried out in the next section. Before doing this, we provide a discussion of the theorems.%

Observe that our lower bound on channel capacity is $\leq 0$ (and thus vacuous) if $|\det Df_w(x)| \leq 1$ for all $(x,w)$. Recall that the determinant of a square matrix represents the volume of the unit cube after it is acted on by the matrix. As such, \cref{thm;NewTheorem1} is only interesting if the system is volume-expansive on some regions of the state space. This is intuitive, since if $f$ is nowhere volume-expansive, it may be possible for the uncontrolled system to have desirable stability properties.%

The results obtained here are consistent with those obtained using information-theoretic techniques in \cite{yuksel2016stationary}, but are in fact a strict refinement. A similar converse result on channel capacity was obtained in \cite{yuksel2016stationary} under the stronger stability criterion of positive Harris recurrence of the closed-loop stochastic process. It reads as follows:%

\begin{theorem} \label{thm;YukselTh}
(\cite[Thm.~4.2]{yuksel2016stationary}) Consider the system%
\begin{equation*}
  x_{t+1} = f(x_t,w_t) + u_t%
\end{equation*}
and suppose that the following assumptions hold:%
\begin{enumerate}[(i)]
\item For any fixed $w$, the function $f_w:\R^N \rightarrow \R^N$ is a $C^1$-diffeomorphism.%
\item There exist $L,M \in \R$ such that $L \leq \log|\det Df_w(x)| \leq M$ for all $x,w \in \R^N$.%
\end{enumerate}
Suppose that a stationary coding and control policy (see \cite{yuksel2016stationary} for a precise definition) is adopted so that under this policy%
\begin{enumerate}[(i)]
\item the Markovian system state and encoder state is positive Harris recurrent (which implies the existence of a unique invariant measure).%
\item $\limsup_{t \rightarrow \infty}h(x_t)/t \leq 0$ \nicolasg{where $h(\cdot)$ denotes the differential entropy}.%
\end{enumerate}
Then the channel capacity must satisfy%
\begin{equation*}
	\int \int \log|\det Df_w(x)|\diff Q(x) \diff \nu(w) \leq C.%
\end{equation*}
\end{theorem}

Let us now compare \cref{thm;NewTheorem1} and \cref{thm;NewTheorem2} with \cref{thm;YukselTh}. \cref{thm;NewTheorem1} is more general in the sense that it applies to arbitrary causal coding and control policies, not just Markov ones. Moreover, it does not require the assumption of sublinear growth of the differential entropy of the state process. \cref{thm;YukselTh} assumes that the state process is positive Harris recurrent which implies unique ergodicity, while \cref{thm;NewTheorem1} only assumes ergodicity of the AMS measure. On the other hand, compared with \cref{thm;NewTheorem2}, \cref{thm;YukselTh} considers a more general class of channels (involving memory) as well as systems taking values in higher dimensions.%

\section{Proofs}\label{sec:pfs}

In this section, we prove our two main theorems. We begin by generalizing the notion of stabilization entropy and proving a technical lemma.%

\subsection{Generalizing stabilization entropy}

Consider system \eqref{mainsys} with a fixed (open-loop) control sequence $u \coloneqq (u_t)_{t \in \Z_+}$, a noise realization $w \coloneqq (w_t)_{t \in \Z_+}$ and an initial state $x_0 \in \R^N$. For such a setup, the trajectory $x \coloneqq (x_t)_{t \in \Z_+} \in (\R^N)^{\Z_+}$ of the state is uniquely determined. Let us denote this trajectory by $\phi(\cdot,x_0,u,w)$ so that for any $t \in \Z_+$, $x_t = \phi(t,x_0,u,w)$.%

We want to find a subset of control sequences that allow to render certain subsets of the state space invariant in a probabilistic sense. This leads to the next definitions of spanning sets and stabilization entropy for finite collections of subsets of $\R^N$ and $\mathds{W}$, respectively, which generalize similar notions in \cite{kawan2019invariance}, where a single set was considered.%

\begin{definition}
Let $B \in \sB(\R^N)$ and $D \in \sB(\mathds{W})$ be finite disjoint unions of Borel sets $B_1,\ldots,B_n$ and $D_1,\ldots,D_m$, respectively. Let also $R$ denote a collection of numbers $r_{k,l} \in [0,1]$ for $k \in \{1,\ldots,n\}$ and $l \in \{1,\ldots,m\}$ satisfying%
\begin{equation*}
  1 - r := \sum_{k=1}^n \sum_{l=1}^m (1 - r_{k,l}) \in [0,1].%
\end{equation*}
Fix $T \in \N$ and $\rho \in (0,1)$. A set of control sequences $S \subseteq (\R^N)^T$ is called $(T,B,D,\rho,R)$-spanning if there exists $\tilde{\Omega} \in \sF$ such that the following conditions hold:%
\begin{itemize}
\item $P(\tilde{\Omega}) \geq 1-\rho$.%
\item For each $\omega \in \tilde{\Omega}$, there exists a control sequence $u \in S$ such that
\begin{equation*}
  \frac{1}{T}|\{t \in [0;T-1] : (\phi(t,x_0(\omega),u,w(\omega)),w_t(\omega)) \in B_k \times D_l \}| \geq 1 - r_{k,l}%
\end{equation*}
holds for all $k$ and $l$.
\end{itemize}
\end{definition}

Note that we abuse notation by calling a set $(T,B,D,\rho,R)$-spanning instead of $(T,(B_k)_{k=1}^{n},(D_l)_{l=1}^m,\rho,R)$-spanning. When doing so, there is the underlying assumption that the partitions of $B$ and $D$ are fixed. No confusion should arise, since we explicitly define the partitions whenever we use the definition.%

In the above definition, the fact that all random variables are modeled on a common probability space ensures that given $\omega$, the initial state and the noise sequence of length $T$ are deterministic. Intuitively speaking, a subset of control sequences of length $T$ is $(T,B,D,\rho,R)$-spanning if the probability that, for all $k,l$, we can maintain the state variable in $B_k$ and the noise variable in $D_l$ for at least $1 - r_{k,l}$ percent of the time, is at least $1 - \rho$. We want to use the size of spanning sets to quantify the difficulty of a control task, which leads to the next definition.%

\begin{definition}
For the system (\ref{mainsys}), we define the $(B,D,\rho,R)$-stabilization entropy by%
\begin{equation*}
  h(B,D,\rho,R) := \limsup_{T \to \infty}\frac{1}{T} \log s(T,B,D,\rho,R),%
\end{equation*}
where $s(T,B,D,\rho,R)$ denotes the smallest cardinality of a $(T,B,D,\rho,R)$-spanning set. We define this quantity to be $\infty$ if no or no finite spanning set exists.%
\end{definition}

It is obvious that finite $(T,B,D,\rho,R)$-spanning sets need not exist. As we will see however, they do exist in desired scenarios.%

The following lemma is instrumental to prove \cref{thm;NewTheorem1}.%

\begin{lemma}	\label{myNewLemma}
Consider system (\ref{mainsys}) with the assumptions of Theorem \ref{thm;NewTheorem1} (i.e., a coding and control policy exists over a noiseless channel of capacity $C = \log|\sM|$ which makes the state process AMS ergodic \nicolasg{with AMS mean $Q$). We recall that $Q$ is stationary and can be projected unambiguously to obtain a measure on $\sB(\R^N)$. Abusing notation we also denote the projection by $Q$}. Let now%
\begin{itemize}
\item $B \coloneqq \bigsqcup_{k=1}^{n} B_k \in \sB(\R^N)$ and $D \coloneqq \bigsqcup_{l=1}^{m} D_l \in \sB(\mathds{W})$ be finite disjoint unions of Borel sets,%
\item $\rho \in (0,1)$ be arbitrary.%
\end{itemize}
Next, define the sequence of numbers $R_\epsilon \coloneqq (r_{k,l})_{1 \leq k \leq n, 1 \leq l \leq m}$, where%
\begin{align*}
	r_{k,l} := \begin{cases}
      (1+\epsilon)(1 - Q(B_k)\nu(D_l)) & \text{ if } Q(B_k)\nu(D_l) \in (0,1) \\
      1 & \text{ if } Q(B_k)\nu(D_l) = 0	\\
      \epsilon & \text{ if } Q(B_k)\nu(D_l) = 1
    \end{cases}
\end{align*}
and observe that for $\epsilon > 0$ small enough, the following conditions are satisfied:%
\begin{enumerate}[(i)]
\item $1 - r \coloneqq \sum_{k=1}^{n}\sum_{l=1}^m(1 - r_{k,l}) \in [0,1]$.%
\item $1 - (1 + \epsilon)(1 - Q(B_k)\nu(D_l)) \in (0,1) $ for all $k,l$ with $Q(B_k)\nu(D_l) \in (0,1)$.%
\end{enumerate}
Thus, for such a small $\epsilon$, the generalized stabilization entropy $h(B,D,\rho,R_\epsilon)$ is well-defined. (Of course, $r$ and the $r_{k,l}$'s are $\epsilon$-dependent, but we drop this from the notation.) Then for all $\epsilon > 0$ further small enough the capacity must satisfy%
\begin{equation}\label{eq_centbound}
  h(B,D,\rho,R_\epsilon) \leq C.%
\end{equation}
\end{lemma}

\begin{proof}
We distinguish two cases.%

\textbf{Case 1:} We can remove the trivial sets with zero measure from the collections $\{B_k\}$ and $\{D_l\}$ and thus assume that $Q(B_k)\nu(D_l) > 0$ for all $(k,l)$. Indeed, if a spanning set can be found for the new collections, it is still spanning for the original ones. If $Q(B_k)\nu(D_l) = 1$ for some $(k,l)$, all the other Cartesian products have measure zero and we can remove them from the collection. Hence, this case reduces to the analysis of a single set as worked out in \cite{kawan2019invariance}, where we considered AMS instead of AMS ergodicity as the control objective. Since AMS ergodicity implies AMS, and $h(B,D,\rho,R)$ reduces to the stabilization entropy notion used in \cite{kawan2019invariance} in case of a single set, the desired inequality follows.%

\textbf{Case 2:} We continue by considering the case where $Q(B_k)\nu(D_l) \in (0,1)$ for all $k,l$. Let $\epsilon > 0$ be small enough such that conditions (i) and (ii) are satisfied and $\epsilon < \rho$. We will show that for any such $\epsilon$ the claim holds.%

Let us denote the process measure by $\mu$, which is AMS by assumption. Let us consider some Borel set $C \subset \R^N$ and let $f:(\R^N)^{\Z_+} \rightarrow \R$ be defined by $f((x_t)_{t\in\Z_+}) := \1_{C}(x_0)$. It is obvious that this function is in $L^1((\R^N)^{\Z_+})$ (with either $Q$ or $\mu$ as the measure). Recalling our ergodicity assumption, the pointwise ergodic theorem tells us that%
\begin{equation*}
	\frac{1}{N} \sum_{j=0}^{N-1} f \circ \theta^j  \xrightarrow[N \rightarrow \infty]{Q-a.s.} \int f \diff Q = \int \1_{C}(x) \diff Q(x) = Q(C).%
\end{equation*}
Crucially however, the above convergence also happens $\mu$-almost surely (see (\ref{def_eq}) or \cite[Lem.~7.5]{gray2009probability}). Now, for any $V \in \sB(\W)$, it is clear by the i.i.d.~property that%

\begin{equation*}
  P\Bigl(\Bigl\{\omega \in \Omega : \lim_{T \to \infty} \frac{1}{T} \sum_{t=0}^{T-1}\1_V(w_t(\omega)) = \nu(V)\Bigr\}\Bigr) = 1.%
\end{equation*}
As such, noting that $x_t$ and $w_t$ are independent at each time step $t$, it follows that%
\begin{equation*}
  P\Bigl(\Bigl\{\omega \in \Omega : \lim_{T \to \infty} \frac{1}{T} \sum_{t=0}^{T-1}\1_{B_k}(x_t(\omega))\1_{D_l}(w_t(\omega)) = Q(B_k)\nu(D_l),\ \forall k,l\Bigr\}\Bigr) = 1.%
\end{equation*}
Let us denote the full measure set, where this convergence happens, by $\hat{\Omega}$.%

We continue by defining the events%
\begin{align*}
  E_i^j &\coloneqq \Bigl\{\omega \in \Omega :\ \Bigl| \frac{1}{T} \sum_{t=0}^{T-1} \1_{B_k}(x_t(\omega))\1_{D_l}(w_t(\omega)) - Q(B_k)\nu(D_l) \Bigr| < \frac{1}{i}\\
	&\qquad\qquad\qquad\qquad\qquad\qquad\qquad\qquad\qquad\qquad\qquad \forall k,l \text{ whenever } T \geq j \Bigr\},\\
  E &\coloneqq \bigcap_{i = 1}^{\infty} \bigcup_{j = 1}^{\infty}E_i^j.%
\end{align*}
It is not hard to see that $\hat{\Omega} \subseteq E$, hence $P(E) = 1$. Furthermore, observe that $E$ is an infinite intersection of ``decreasing'' sets (in the containment sense). Hence,%
\begin{equation*}
  P\bigg(\bigcup_{j = 1}^{\infty}E_i^j\bigg) = 1 \mbox{\quad for all\ } i \in \N.%
\end{equation*}
Let now $I_0$ be large enough such that%
\begin{equation*}
  \frac{1}{I_0} \leq \epsilon (1 - Q(B_k)\nu(D_l))  \mbox{\quad for all\ } k \in \{1,\ldots,n\},\ l \in \{1,\ldots,m\}%
\end{equation*}
and observe that $E_{I_0}^1 \subseteq E_{I_0}^2 \subseteq	E_{I_0}^3 \subseteq \cdots$. By continuity of probability, we have%
\begin{equation*}
	\lim_{j \rightarrow \infty} P(E_{{I_0}}^{j}) = P\bigg(\bigcup_{j = 1}^{\infty}E_{I_0}^j\bigg) = 1,%
\end{equation*}
and thus there exists $J_0$ such that $P(E_{I_0}^{j}) \geq 1 - \epsilon$ for all $j\geq J_0$. For an arbitrary $T \geq J_0$, we define the set of control sequences%
\begin{equation*}
  S_T := \{u_{[0;T-1]}(\omega) : \omega \in E_{I_0}^T\}.%
\end{equation*}
We claim that this set is $(T,B,D,\rho,R_\epsilon)$-spanning. We use the set $\tilde{\Omega}_T := E_{I_0}^T \in \sF$ to show this, where we note that it satisfies $P(\tilde{\Omega}_T) \geq 1 - \epsilon > 1 - \rho$, as required. For every $\omega \in \tilde{\Omega}_T$ and all $k,l$, the control sequence $u_{[0;T-1]}(\omega)$ results in the joint state-noise process satisfying%
\begin{align}\label{contradiction2}
  \Big|\frac{1}{T} \sum_{t=0}^{T-1} \1_{B_k}(x_t(\omega))\1_{D_l}(w_t(\omega)) - Q(B_k)\nu(D_l)\Big| < \frac{1}{I_0} \leq \epsilon (1 - Q(B_k)\nu(D_l)).%
\end{align}
To prove the claim, it now suffices to show that for all $\omega \in \tilde{\Omega}_T$ and $k,l$ we have%
\begin{align*}	
&	\frac{1}{T}|\{t \in [0;T-1] : (\phi(t,x_0(\omega),u_{[0;T-1]}(\omega),w(\omega)),w_t(\omega)) \in B_k \times D_l \}|\\
	&	\qquad\qquad\qquad \geq 1 - (1+\epsilon)(1 - Q(B_k)\nu(D_l)) = (1 + \epsilon) Q(B_k)\nu(D_l) - \epsilon.%
\end{align*}
This follows directly from \eqref{contradiction2}. Also, since the coding and control policy can generate at most $|\sM|^T$ distinct control sequences by time $T$, it follows that $|S_T| \leq |\sM|^T$, therefore $s(T,B,D,\rho,R_\epsilon) \leq |\sM|^T$. Recalling that $T \geq J_0$ was arbitrary, we find that%
\begin{equation*}
	\log s(T,B,D,\rho,R_\epsilon) \leq T \log|\sM| = TC \mbox{\quad for all\ } T \geq J_0,%
\end{equation*}
and therefore dividing by $T$ and letting $T\rightarrow\infty$ yields the desired capacity bound \eqref{eq_centbound}, which completes the proof.
\end{proof}

\subsection{Proof of \cref{thm;NewTheorem1}}

\begin{proof}
Let $c \in (0,1)$ be such that $c < |\det Df_w(x)|$ for all $x \in \R^N$ and $w \in \mathds{W}$. Let also $\delta > 0$ and $\rho \in (0,1)$ be arbitrary. Next, fix a partition of a Borel set $B \subset \R^N$ and let $D = \mathds{W}$, respectively; let $(B_k)_{k=1}^{n}$ be a partition of $B$ and $(D_l)_{l=1}^{m}$ a partition of $D$. Suppose that $B$ has finite Lebesgue measure and%
\begin{equation*}
  Q(B) > 1 - \frac{\delta}{2|\log c|},%
\end{equation*}
where $Q$ denotes the asymptotic mean of the state process. Let $\epsilon > 0$ be small enough such that \cref{myNewLemma} holds, resulting in%
\begin{equation*} 
  h(B,D,\rho,R_\epsilon) \leq C,%
\end{equation*}
where $R_\epsilon$ is the associated collection of $r_{k,l}$'s as defined in \cref{myNewLemma}. Let also $1 - r \coloneqq \sum (1-r_{k,l})$. It is easy to see that $r = 1 - (1 + \epsilon)Q(B) + n m \epsilon$ (or $r = \epsilon$ if one of the $B_k \times D_l$ has full $Q \times \nu$-measure) thus we see that for every sufficiently small $\epsilon$,%
\begin{equation}\label{eq_r_choice}
  2r < \frac{\delta}{|\log c|}.%
\end{equation}
Now fix a sufficiently large $T\in\N$ and let $S$ be a finite $(T,B,D,\rho,R_\epsilon)$-spanning set (whose existence is guaranteed by the proof of \cref{myNewLemma}) with $\tilde{\Omega} \in \sF$, $P(\tilde{\Omega}) \geq 1 - \rho$, the associated subset of $\Omega$. Also let%
\begin{align*}
A &\coloneqq \{(w(\omega),x_0(\omega)) : \omega \in \tilde{\Omega}\}, \\
A(u) &\coloneqq \{ (w,x) \in \mathds{W}^{\Z_+} \times \R^N : \frac{1}{T}\sum_{t=0}^{T-1}\1_{B_k \times D_l}(\phi(t,x,u,w),w_t) \geq 1 - r_{k,l},\ \forall k,l \}\\
A(u,w) &\coloneqq \{ x \in \R^N : (w,x) \in A(u) \}%
\end{align*}
and observe that%
\begin{equation}\label{eq_nc1}
  A \subseteq \bigcup_{u \in S}A(u).%
\end{equation}
By the theorem of Fubini-Tonelli, we have%
\begin{equation}\label{eq_nc2}
  (\nu \times m)(A(u)) = \int m(A(u,w)) \diff \nu(w).%
\end{equation}
Let us now define a set consisting of disjoint collections of subsets of $\{0,\ldots,T-1\}$:%
\begin{align*}
  \A &\coloneqq \{ \Lambda = \{\Lambda_k^l\}_{k,l} : \bigsqcup_{k=1}^{n}\bigsqcup_{l=1}^{m}\Lambda_k^l \subseteq \{0,\ldots,T-1\},\\
	&\qquad\qquad |\Lambda_k^l| \geq (1 - r_{k,l})T, \forall k=1,\ldots,n,\ l=1,\ldots,m \}%
\end{align*}
and note that as a consequence of the definition, $|\bigsqcup_{k=1}^{n}\bigsqcup_{l=1}^{m}\Lambda_k^l| \geq (1-r)T$ for all $\Lambda \in \A$. For such a $\Lambda$, define the set%
\begin{equation*}
  A(u,w,\Lambda) := \{x \in \R^N : (\phi(t,x,u,w),w_t) \in B_k \times D_l \Leftrightarrow t \in \Lambda_k^l \text{ for all } k,l\}%
\end{equation*}
and also (writing $\varphi_{t,u,w}(\cdot) := \varphi(t,\cdot,u,w)$)%
\begin{equation*}
  A_t(u,w,\Lambda) := \phi_{t,u,w}(A(u,w,\Lambda)),\quad t = 0,1,\ldots,T-1.%
\end{equation*}
It is not hard to see that $A(u,w) = \bigsqcup_{\Lambda \in \A} A(u,w,\Lambda)$ is a disjoint union, implying%
\begin{equation}\label{eq_nc3} 
  m(A(u,w)) = \sum_{\Lambda \in \A} m(A(u,w,\Lambda)).%
\end{equation}
If $M>0$ is an upper bound for the density of $\pi_0$, it follows that%
\begin{equation}\label{eq_nc4} 
  1 - \rho \leq (\nu \times \pi_0)(A) \leq M\cdot(\nu \times m)(A).%
\end{equation}
We also have%
\begin{equation*}
	A_t(u,w,\Lambda) \subseteq B_k \text{\quad whenever}\ t \in \Lambda_{k,l}, \ \forall \ k \in \{1,\ldots,n\},\ l \in \{1,\ldots,m\}.%
\end{equation*}
Next, we define the following numbers:%
\begin{align*}
  c_{k,l} \coloneqq \inf_{(x,w) \in B_k \times D_l}|\det Df_w(x)|.%
\end{align*}
Recalling the fact that $f_w$ is injective and $C^1$, for all $(k,l)$ we have%
\begin{align*}
	m(A_{t+1}(u,w,\Lambda)) &\geq c_{k,l} \cdot m(A_t(u,w,\Lambda)) \text{ whenever } t\in \Lambda_{k,l}, \\
	m(A_{t+1}(u,w,\Lambda)) &\geq c \cdot m(A_t(u,w,\Lambda)) \text{ whenever } t \notin \bigsqcup \Lambda_{k,l}.
\end{align*}
Letting $t^*(\Lambda_{k,l}) \coloneqq \max\Lambda_{k,l}$, $t^*(\Lambda) \coloneqq \max_{k,l}t^*(\Lambda_{k,l})$ and applying the above inequalities repeatedly, it is not hard to see that%
\begin{align*}
  m(A(u,w,\Lambda)) \Bigl( \prod_{k=1}^{n}\prod_{l=1}^{m} c_{k,l}^{|\Lambda_{k,l}|-1} \Bigr) c^{rT+nm} \leq m(A_{t^*(\Lambda)}(u,w,\Lambda)).%
\end{align*}
Recall that $c \leq c_{k,l}$. Now in principle, all the exponents of the $c_{k,l}$'s should be $|\Lambda_{k,l}|$, except for possibly one which should be $|\Lambda_{k,l}| - 1$. We do not know which one though, so we write the weaker inequality as above. Combining this with \eqref{eq_nc1}, \eqref{eq_nc2}, \eqref{eq_nc3} and \eqref{eq_nc4}, we obtain%
\begin{align*}
 \frac{1}{M}(1 - \rho) &\leq (\nu \times m)(A) \\
  &\leq |S| \max_{u\in S} (\nu \times m)(A(u)) \\
  &= |S| \max_{u\in S} \int m(A(u,w))\diff \nu(w) \\
  &= |S| \max_{u\in S} \int \sum_{\Lambda \in \A} m(A(u,w,\Lambda))\diff \nu(w) \\
  &= |S| \max_{u\in S} \sum_{\Lambda \in \A} \int m(A(u,w,\Lambda))\diff \nu(w) \\
  &\leq |S| \max_{u\in S} \sum_{\Lambda \in \A} \int m(A_{t^*(\Lambda)}(u,w,\Lambda)) c^{-(rT+nm)}\prod_{k = 1}^{n} \prod_{l = 1}^{m} c_{k,l}^{-(|\Lambda_{k,l}|-1)}  \diff \nu(w) \\
  &= |S| \cdot c^{-(rT+nm)} \max_{u\in S} \sum_{t_{1,1} = (1-r_{1,1})T}^{T} \cdots \sum_{t_{n,m} = (1-r_{n,m})T}^{T} \\
	&\quad \int \sum_{\Lambda \in \A :\ t^*(\Lambda_{k,l}) = t_{k,l} \forall k,l} m(A_{t^*(\Lambda)}(u,w,\Lambda)) \prod_{k = 1}^{n} \prod_{l=1}^{m} c_{k,l}^{-(|\Lambda_{k,l}|-1)} \diff \nu(w)  \\ 
  &\leq |S| \cdot c^{-(2rT+nm)} \max_{u\in S} \sum_{t_{1,1} = (1-r_{1,1})T}^{T} \cdots \sum_{t_{n,m} = (1-r_{n,m})T}^{T}\\
	&\quad \int \sum_{\Lambda \in \A :\ t^*(\Lambda_{k,l}) = t_{k,l} \forall k,l} m(A_{t^*(\Lambda)}(u,w,\Lambda)) \prod_{k = 1}^{n} \prod_{l=1}^{m} c_{k,l}^{-((1-r_{k,l})T-1)} \diff \nu(w).%
\end{align*}
In the last inequality we use that%
\begin{align*}
  &c^{rT + nm} \prod_{k,l} c_{k,l}^{|\Lambda_{k,l}|-1} = c^{rT + \sum_{k,l}|\Lambda_{k,l}|} \prod_{k,l} \Bigl(\frac{c_{k,l}}{c}\Bigr)^{|\Lambda_{k,l}|-1}\\
	&\geq c^{rT + \sum_{k,l}|\Lambda_{k,l}|} \prod_{k,l} \Bigl(\frac{c_{k,l}}{c}\Bigr)^{(1-r_{k,l})T-1} = c^{rT + \sum_{k,l}|\Lambda_{k,l}| - (1-r)T + nm} \prod_{k,l}c_{k,l}^{(1-r_{k,l})T-1}\\
	&\geq c^{2rT + nm} \prod_{k,l}c_{k,l}^{(1-r_{k,l})T-1}.%
\end{align*}
Observe that the sets $A_{t^*(\Lambda)}(u,w,\Lambda)$ with $\Lambda \in \A$, $t^*(\Lambda)$ fixed, are pairwise disjoint, since they are the images of the corresponding sets $A(u,w,\Lambda)$ under the injective map $\varphi_{t^*(\Lambda),u,w}$. Moreover, all of these sets are contained in $B$. Hence,%
\begin{equation*}
  \sum_{\Lambda \in \A : t^*(\Lambda_{k,l}) = t_{k,l} \forall k,l} m(A_{t^*(\Lambda)}(u,w,\Lambda)) \leq m(B),%
\end{equation*}
which, together with the above chain of inequalities, implies%
\begin{align*}
 \frac{1}{M}(1 - \rho) \leq  |S| \cdot m(B) \cdot c^{-(2rT+nm)} \cdot \prod_{k = 1}^{n} \prod_{l=1}^{m} c_{k,l}^{-((1-r_{k,l})T-1)}  \prod_{k=1}^n\prod_{l=1}^m (r_{k,l}T + 1).%
\end{align*}
Since this inequality holds for every $T$ sufficiently large, we can take logarithms on both sides, divide by $T$ and let $T \rightarrow \infty$. This results in%
\begin{equation*}
  0 \leq h(B,D,\rho,R_{\epsilon}) - 2r \log c - \sum_{k=1}^n \sum_{l=1}^n (1-r_{k,l}) \log c_{k,l}.%
\end{equation*}
Recalling the definition of $r_{k,l}$, the fact that $\epsilon$ can be chosen arbitrarily small and \eqref{eq_r_choice}, this leads to the estimate%
\begin{equation*}
  C + \delta \geq \sum_{k=1}^n \sum_{l=1}^n Q(B_k)\nu(D_l) \inf_{(x,w) \in B_k \times D_l}\log |\det Df_w(x)|.%
\end{equation*}
Considering the supremum of the right-hand side over all finite measurable partitions of $B$ and $\mathds{W}$ leads to%
\begin{equation*}
  C + \delta \geq \int \int \1_B(x) \log |\det Df_w(x)| \diff Q(x) \diff \nu(w),%
\end{equation*}
where we use that the integrand is uniformly bounded below by $\log c$ (and hence, we can assume that it is non-negative). Considering now an increasing sequence of sets $B_k \subset \R^N$ whose union is $\R^N$, we can invoke the theorem of monotone convergence to obtain the desired estimate, observing that $\delta$ can be made arbitrarily small as $B_k$ becomes arbitrarily large.
\end{proof}

\subsection{Proof of \cref{thm;NewTheorem2}}

\begin{proof}
Suppose for a contradiction that a causal coding and control policy is such that the state process is AMS ergodic, but that the converse of inequality (\ref{Theorem2Ineq}) holds. Let $r>0$ be small enough so that%
\begin{equation*}
  C < (1-3r)\int  \int \log|f_w'(x)| \diff Q(x) \diff \nu(w).%
\end{equation*}
Since we can approximate the integral by the integral over associated step functions, for any $b \in \N$ large enough, there exists a disjoint collection of intervals $B_1,\ldots,B_{2^{b+1}}$ and a partition $D_1,\ldots,D_m$ of $\W$ such that $B \coloneqq [-b,b] = \bigsqcup_{k=1}^{2^{b+1}}B_k$, and%
\begin{equation*}
  C < (1-3r)\sum_{l=1}^{m}\sum_{k=1}^{2^{b+1}} \nu(D_l)Q(B_k) \log c_{k,l},%
\end{equation*}
where $c_{k,l} \coloneqq \inf_{(x,w) \in B_k \times D_l} |f_w'(x)|$. Put $n \coloneqq 2^{b+1} + 1$, and fix a $b$ (and the associated collection $(B_k)_{k=1}^{n-1}$ of intervals) further large enough such that%
\begin{equation}  \label{QBbound}
  Q([-b,b])(1-r) > 1 - \frac{2.5}{2}r%
\end{equation}
which is possible by continuity of probability. Finally, let $B_n := \R \setminus \bigsqcup_{k=1}^n B_k$. For brevity, in the rest of the proof we write%
\begin{equation*}
  m_{k,l} := Q(B_k)\nu(D_l),\quad k = 1,\ldots,n,\ l = 1,\ldots,m.%
\end{equation*}
Next, we define the following sets in a slightly different manner than in the previous proof:%
\begin{align*}
  A_T(u,w) &\coloneqq \{x \in \R : \forall k,l  \text{ and } \forall N \in \{\Ceil{T(1-3r)},\ldots,T\},\\
	                    &\qquad \frac{1}{N}|\{t \in [0;N-1] : (\phi(t,x,u,w),w_t) \in B_k \times D_l \}| \geq m_{k,l}(1-r) \}.%
\end{align*}
It is easy to see that this set is always bounded. Later on, for appropriate parameters, we will also see that the set is nonempty. For these cases, let%
\begin{equation*}
  \overline{A}_T(u,w) \coloneqq [\inf A_T(u,w),\sup A_T(u,w)]%
\end{equation*}
and let $x_0(T,u,w)$ denote the midpoint of this interval. We claim that there exists $T$ larger than some threshold $M_1 = M_1(r)$ so that for all $u,w$ and $x_1,x_2 \in A_T(u,w)$ there exists a $t^*$ with $\Ceil{(1-2.5r)T} \leq t^* \leq T-1$ satisfying%
\begin{equation*}
  \phi(t^*,x_i,u,w) \in B \spc \text{for} \spc i\in\{1,2\}.%
\end{equation*}
To see this, suppose otherwise. Then for at least one $i \in \{1,2\}$ we have%
\begin{align*}
  &|\{t \in [0;T-1] : \phi(t,x_i,u,w) \in B\}| \leq \Ceil{(1-2.5r)T} + \frac{1}{2}(T - \Ceil{(1-2.5r)T}) \\
											&\leq \frac{1}{2}((1-2.5r)T + 1) + \frac{1}{2}T	= \frac{1}{2} + (1 + (1-2.5r))\frac{1}{2}T  \\
											&= \frac{1}{2} + \bigg(1-\frac{2.5}{2}r \bigg)T < (1-r)Q(B)T,%
\end{align*}
where the last inequality holds for $T$ large enough from the assumption \eqref{QBbound} on $Q(B)$.%

This is a contradiction to $x_i \in A_T(u,w)$, which follows by recalling the definition of $A_T(u,w)$. Let now $\epsilon > 0$ and $\delta > 0$ be given. By the pointwise ergodic theorem (see the construction in the proof of \cref{myNewLemma}), there exists an $M_2 \coloneqq M_2(\epsilon,\delta) \in \N$ such that for all $T \geq M_2$%
\begin{align*}
  &P(\{\omega \in \Omega : \forall k,l, \forall N \geq (1-3r)T,\\
&\qquad \spc \frac{1}{N} \sum_{t=0}^{N-1}\1_{B_k}(x_t(\omega))\1_{D_l}(w_t(\omega)) \geq m_{k,l}(1-\delta) \}) > 1 - \epsilon.%
\end{align*}
We denote by $\tilde{\Omega}(\epsilon,\delta,M_2)$ the set of $\omega$'s for which the event within the braces of the above expression occurs. Recalling that $c_{k,l} \coloneqq \inf_{(x,w) \in B_k \times D_l}|\det Df_w(x)|$ and letting $u,w$ and $x_1,x_2 \in A_T(u,w)$ be arbitrary, we have%
\begin{align}	\label{ATBound}
  |x_1 - x_2| \leq \frac{2b}{\prod_{k,l}c_{k,l}^{m_{k,l}(1-\delta)t^*}} \leq \frac{2b}{\prod_{k,l}c_{k,l}^{m_{k,l}(1-\delta)T(1-2.5r)}}%
\end{align}
which follows by noting that%
\begin{align*}
&\prod_{k,l}c_{k,l}^{(1-\delta)m_{k,l}(1-2.5r)T}|x_1 - x_2| \leq \prod_{k,l}c_{k,l}^{(1-\delta)m_{k,l}(1-2.5r)T}(|x_1| + |x_2|) \\ &\leq \prod_{k,l}c_{k,l}^{(1-\delta)m_{k,l}t^*}|x_1| + \prod_{k,l}c_{k,l}^{(1-\delta)m_{k,l}t^*}|x_2| 
\leq |\phi(t^*,x_1,u,w)| + |\phi(t^*,x_2,u,w)| \leq 2b.%
\end{align*}
Given a realization $\omega \in \Omega$, we denote by $x_0(\omega)$ and $w(\omega)$ the resulting realizations of the initial state and noise sequence, respectively. Given these realizations, the control sequence is thus fully determined, and denoted by $u(\omega)$. It follows quite easily that $\omega \in \tilde{\Omega}(\epsilon,\delta,M_2)$ implies $x_0(\omega) \in A_T(u(\omega),w(\omega))$ for all $T \geq M_2$ and all $\delta < r$. Combining this with \eqref{ATBound}, we conclude that%
\begin{equation*}
  |x_0(\omega) - x_0(T,u(\omega),w(\omega))| \leq \frac{b}{\prod_{k,l}c_{k,l}^{(1-\delta)m_{k,l}(1-2.5r)T}}%
\end{equation*}
for every $T \geq M_2(\epsilon,\delta)$ and every $\omega \in \tilde{\Omega}(\epsilon,\delta,M_2)$. Letting $\delta$ be small enough so that both $(1-3r) \leq (1-2.5r)(1-\delta)$ and $\delta < r$ hold, we conclude that%
\begin{equation*}
  \liminf_{T \to \infty} P\Bigl(\Bigl\{\omega \in \Omega : |x_0(\omega) - x_0(T,u(\omega),w(\omega))| \leq \frac{b}{\prod_{k,l}c_{k,l}^{m_{k,l}(1-3r)T}}\Bigr\}\Bigr) \geq 1 - \epsilon%
\end{equation*}
and since $\epsilon > 0$ was also arbitrary, it follows that%
\begin{equation}\label{keyeq3}
  \limsup_{T \to \infty} P\Bigl(\Bigl\{\omega \in \Omega : |x_0(\omega) - x_0(T,u(\omega),w(\omega))| > \frac{b}{\prod_{k,l}c_{k,l}^{m_{k,l}(1-3r)T}}\Bigr\}\Bigr) = 0.%
\end{equation}

We will see that our initial hypothesis leads to a contradiction with the above equation. To this effect, let us choose $\alpha \in (0,1/2)$ small enough so that for all sufficiently large $L$:%
\begin{equation}\label{alphaassumption}
  1 - \frac{\rho_{\min} \cdot (1-\alpha) }{2 \cdot \rho_{\max}} + \frac{\rho_{\max}^2}{2L\rho_{\min}^2} + \frac{2 \cdot \rho_{\max}}{\rho_{\min}} \frac{\alpha}{1-\alpha} < 1.%
\end{equation}
Let also $\tilde{\Omega} \in \sF$ be such that $P(\tilde{\Omega}) > 1 - \alpha$, and such that for all $T$ large enough (say, larger than $C(\alpha)$),%
\begin{align*}
  |x_0(\omega) - x_0(T,u(\omega),w(\omega))| \leq \frac{b}{\prod_{k,l}c_{k,l}^{m_{k,l}(1-3r)T}}
\end{align*}
for all $\omega \in \tilde{\Omega}$. The idea from here on is to treat $\tilde{\Omega}$ as ``the universe'', since conditioning on this set gives the above deterministic bound. We proceed by defining%
\begin{align*}
  U_T &\coloneqq \{(\gamma_0(q_{0}'),\ldots,\gamma_{T-1}(q_{[0;T-1]}')) \in U^T : q_{[0;T-1]}' \in (\sM')^T\}, 	\\
  \tilde{U}_T &\coloneqq \{(\gamma_0(q_{0}'(\omega)),\ldots,\gamma_{T-1}(q_{[0;T-1]}'(\omega))) \in U^T : \omega \in \tilde{\Omega} \}, \\
  \tilde{R} &\coloneqq \limsup_{T \to \infty} \frac{1}{T}\log|\tilde{U}_T|.%
\end{align*}

We now treat two distinct cases: In Case 1, we show that the condition $\tilde{R} < (1-3r)\sum_{k,l}m_{k,l} \log c_{k,l}$ cannot hold if we want to achieve the desired result. This leaves us with Case 2: the condition that $\tilde{R} \geq (1-3r)\sum_{k,l}m_{k,l}\log c_{k,l}$; however, this condition would imply $\tilde{R} > C$. We show that this cannot hold either, through a tedious argument involving a strong converse to channel coding (with feedback) and optimal transport theory. In the following, we study these two cases separately.%

{\bf Case 1:} Let us suppose that%
\begin{equation} \label{Case1Assumption}
  \tilde{R} < (1-3r)\sum_{k,l}m_{k,l}\log c_{k,l}.%
\end{equation}
Let $\epsilon > 0$ be small enough so that $\tilde{R} + 2\epsilon < (1-3r) \sum_{k,l} m_{k,l}\log c_{k,l}$ and observe that for all $T$ large enough,%
\begin{equation}\label{case1eq}
  |\tilde{U}_T| \leq 2^{(\tilde{R} + \epsilon)T}.%
\end{equation}
Recall also that $\tilde{\Omega}$ is such that for all $T$ large enough,%
\begin{equation}\label{bound4}
  |x_0(\omega) - x_0(T,u(\omega),w(\omega))| \leq \frac{b}{\prod_{k,l}c_{k,l}^{m_{k,l}(1-3r)T}} \mbox{\quad for all\ } \omega \in \tilde{\Omega}.%
\end{equation}
We now fix a noise realization $w$. For all $T$ large enough so that \eqref{case1eq} holds,%
\begin{align*}
m\bigg(\bigcup_{u \in \tilde{U}_T} \overline{A}_T(u,w) \bigg) &\leq \frac{2b \cdot 2^{(\tilde{R} + \epsilon)T}}{\prod_{k,l}c_{k,l}^{m_{k,l}(1-3r)T}}  \leq  \frac{2b \cdot 2^{((1-3r) \sum_{k,l} m_{k,l}\log c_{k,l} - \epsilon )T}}{\prod_{k,l}c_{k,l}^{m_{k,l}(1-3r)T}} \\ 
&\leq \frac{2b \cdot 2^{-\epsilon T} \cdot \prod_{k,l}2^{T(1-3r) m_{k,l}\log c_{k,l}}}{\prod_{k,l}c_{k,l}^{m_{k,l}(1-3r)T}} = \frac{2b}{2^{\epsilon T}},%
\end{align*}
where the inequalities follow by applying the union bound, and from \eqref{case1eq} and \eqref{Case1Assumption}. The above yields%
\begin{equation*}
  \lim_{T \to \infty}  m\bigg(\bigcup_{u \in \tilde{U}_T} \overline{A}_T(u,w) \bigg) = 0,%
\end{equation*}
and thus by the absolute continuity and boundedness assumptions on $\pi_0$, we have%
\begin{equation*}
  \lim_{T \to \infty}\pi_0\big(\bigcup_{u\in \tilde{U}_T} \overline{A}_T(u,w) \big)\big) = 0.%
\end{equation*}

On the other hand, let us define $J \coloneqq \{ w \in \W^{\Z_+} : P(\{\omega \in \tilde{\Omega} |w(\omega) = w\}) > 0 \}$. We note that $J$ is the projection of $\tilde{\Omega}$ onto $\R^{\Z_+}$ from which the set $\{w: P(\omega \in \tilde{\Omega} |w(\omega) = w)=0\}$ is taken out; these ensure that $J$ is a universally measurable set since the image of a Borel set under a measurable map is universally measurable \cite{dynkin1979controlled}.%

We can therefore write%
\begin{align*}	
&\limsup_{T \to \infty}P\Bigl(\Bigl\{ \omega \in \Omega : |x_0(\omega) - x_0(T,u(\omega),w(\omega))| \leq \frac{b}{\prod_{k,l}c_{k,l}^{m_{k,l}(1-3r)T}} \big| \omega \in \tilde{\Omega} \Bigr\}\Bigr) \\
&= \limsup_{T \to \infty} \Big( P\Bigl(\Bigl\{ \omega \in \Omega : |x_0(\omega) - x_0(T,u(\omega),w(\omega))| \leq \\
&\qquad\qquad\qquad\qquad   \frac{b}{\prod_{k,l}c_{k,l}^{m_{k,l}(1-3r)T}} \big| \omega \in \tilde{\Omega}, w(\omega) \in J \Bigr\}\Bigr) \cdot P(J)  \\
&+ P\Bigl(\Bigl\{ \omega \in \Omega : |x_0(\omega) - x_0(T,u(\omega),w(\omega))| \leq \\
&\qquad\qquad\qquad\qquad  \frac{b }{\prod_{k,l}c_{k,l}^{m_{k,l}(1-3r)T}} \big| \omega \in \tilde{\Omega}, w(\omega) \in J^c \Bigr\}\Bigr) \cdot P(J^c) \Big).%
\end{align*}
Now, noting that $P(\tilde{\Omega}) > 1 - \alpha$ implies $\nu(J^c) \leq \alpha$, we can further write%
\begin{align*}
  &\leq \limsup_{T \to \infty}P\Bigl(\Bigl\{ \omega \in \Omega : |x_0(\omega) - x_0(T,u(\omega),w(\omega))| \leq \\
		&\qquad\qquad\qquad\qquad \frac{b }{\prod_{k,l}c_{k,l}^{m_{k,l}(1-3r)T}} \big| \omega \in \tilde{\Omega}, w(\omega) \in J \Bigr\}\Bigr) \cdot P(J) + \alpha.%
\end{align*}
Observe that for a noise realization $w \in J$, we have%
\begin{align*}
&\limsup_{T \to \infty}P\Bigl(\Bigl\{ \omega \in \Omega : |x_0(\omega) - x_0(T,u(\omega),w(\omega))| \leq \\
&\qquad\qquad\qquad\qquad  \frac{b }{\prod_{k,l}c_{k,l}^{m_{k,l}(1-3r)T}} \big| \omega \in \tilde{\Omega},w(\omega) = w \Bigr\}\Bigr) \\
&\leq \limsup_{T \to \infty}P\Bigl(\Bigl\{\omega \in \Omega : x_0(\omega) \in \bigcup_{u \in \Tilde{U}_T}\overline{A}_T(u,w) \big| \omega \in \tilde{\Omega},w(\omega)=w \Bigr\}\Bigr) \\
&\leq \frac{1}{P(\omega \in \tilde{\Omega}| w(\omega) = w)}\limsup_{T \to \infty}P\Bigl(\Bigl\{\omega \in \Omega : x_0(\omega) \in \bigcup_{u \in \Tilde{U}_T}\overline{A}_T(u,w)|w(\omega) = w \Bigr\}\Bigr)	\\ 
&= \frac{1}{P(\omega \in \tilde{\Omega}|w(\omega) = w)}\limsup_{T \to \infty}\pi_0\Bigl(\bigcup_{u \in \Tilde{U}_T}\overline{A}_T(u,w)\Bigr) = 0,%
\end{align*}
where the first inequality can be justified by noting that%
\begin{equation*}
  |x_0(\omega) - x_0(T,u(\omega),w(\omega))| \leq \frac{b}{\prod_{k,l}c_{k,l}^{m_{k,l}(1-3r)T}} \quad\Rightarrow\quad x_0(\omega) \in \overline{A}_T(u(\omega),w)%
\end{equation*}
for all $T$ sufficiently large (see (\ref{bound4})) and the last inequality follows by independence of noise and initial state. We thus have a uniform upper bound on the limsup when conditioned on $w \in J$, hence%
\begin{align*}
&\limsup_{T \to \infty}P\Bigl(\Bigl\{ \omega \in \Omega : |x_0(\omega) - x_0(T,u(\omega),w(\omega))| \leq \\
&\qquad\qquad\qquad\qquad \frac{b}{\prod_{k,l}c_{k,l}^{m_{k,l}(1-3r)T}} \big| \omega \in \tilde{\Omega}, w(\omega) \in J \Bigr\}\Bigr) = 0.%
\end{align*}
Therefore,%
\begin{equation*}
  \limsup_{T \to \infty}P\Bigl(\Bigl\{\omega \in \Omega: |x_0(\omega) - x_0(T,u(\omega),w(\omega))| \leq \frac{b}{\prod_{k.l}c_{k,l}^{m_{k,l}(1-3r)T}} \big| \omega \in \tilde{\Omega} \Bigr\}\Bigr) \leq \alpha,%
\end{equation*}
which contradicts \eqref{keyeq3}, since $\alpha < 1/2$. Hence, the proof for Case 1 is complete.%

{\bf Case 2:} Now we suppose that%
\begin{equation*}
  \tilde{R} \geq (1-3r)\sum_{k,l}m_{k,l}\log c_{k,l},%
\end{equation*}
thus by assumption we also have $\tilde{R} > C$. Recall that the proof is by contradiction. In this case, we will obtain a contradiction to a generalized version of the strong converse theorem for discrete memoryless channels with feedback (see \cite{kemperman1971strong} and \cref{thm;strongconv}). Recall that by definition of $\tilde{\Omega}$, we have that for any $T$ sufficiently large, the inequality%
\begin{equation*}
  |x_0(\omega) - x_0(T,u(\omega),w(\omega))| \leq \frac{b}{\prod_{k,l}c_{k,l}^{m_{k,l}(1-3r)T}}%
\end{equation*}
holds for any $\omega \in \tilde{\Omega}$. Also recall that $P(\tilde{\Omega}) > 1 - \alpha$ for $\alpha$ satisfying the important assumption  \eqref{alphaassumption}. As such, there must exist some noise realization $w$ such that $P(\{\omega \in \tilde{\Omega}|w(\omega) = w\}) > 1 - \alpha$. This can be seen by contradiction; suppose no such realization exists. Letting $\nu$ denote the measure on the space of noise realizations, we can write%
\begin{align}
  P(\omega \in \tilde{\Omega}) = \int P(\omega \in \tilde{\Omega}|w(\omega) = \tilde{w}) \diff \nu(\tilde{w}) \leq  \int (1 - \alpha) \diff \nu(\tilde{w}) = 1-\alpha%
\end{align}
which is a contradiction since $P(\tilde{\Omega})> 1-\alpha$. The existence of such a realization $w$ yields%
\begin{align}\label{conditionedbound}
\begin{split}
&\liminf_{T \to \infty}P\Bigl(\Bigl\{\omega \in \Omega : |x_0(\omega) - x_0(T,u(\omega),w)| \leq\\
 &\qquad\qquad\qquad\qquad \frac{b}{\prod_{k,l}c_{k,l}^{m_{k,l}(1-3r)T}} \big| w(\omega) = w \Bigr\}\Bigr) > 1 - \alpha.%
\end{split}
\end{align}
In the remainder of the proof, we condition on the occurrence of the noise realization $w$. We follow an almost identical approach as in the proof from \cite{kawan2019invariance}; we will construct a sequence of codes to transmit a uniform random variable which contradicts a version of the strong converse result for DMCs. This is accomplished in four steps.%

\textbf{Step 1 (Construction of bins):} For every $T \geq 1$, define $S_T \coloneqq \{x_0(T,u,w) : u \in \tilde{U}_T \}$ and enumerate the elements of this set so that%
\begin{equation}\label{enum}
  S_T \coloneqq \{x_1(T),\ldots,x_{n_1(T)}(T)\}.%
\end{equation}
We continue by defining the not necessarily disjoint collection of bins%
\begin{equation*}
  B_{i}^{T} := \Bigl\{x \in \R : |x - x_i(T)| \leq \frac{b}{\prod_{k,l}c_{k,l}^{m_{k,l}(1-3r)T}}\Bigr\},\quad i = 1,\ldots,n_1(T).%
\end{equation*}
Note that for a fixed $T$, each bin has the same Lebesgue measure which we denote by $\rho_T \coloneqq (2b)/\prod_{k,l}c_{k,l}^{m_{k,l}(1-3r)T}$. Recalling that $P(\{\omega \in \tilde{\Omega}|w(\omega) = w\}) > 1 - \alpha$, it follows that%
\begin{equation*}
  1 - \alpha < \liminf_{T \to \infty}P\Bigl(\Bigl\{\omega \in \Omega: x_0(\omega) \in \bigcup_{i=1}^{n_1(T)}B_i^T \big| w(\omega) = w \Bigr\}\Bigr),%
\end{equation*}
from which by independence of noise and initial state, we obtain%
\begin{equation}\label{eq_bit_measure}
  1 - \alpha < \liminf_{T \to \infty} \pi_0 \Bigl( \bigcup_{i=0}^{n_1(T)}B_i^T \Bigr).%
\end{equation}
We will disregard the bins that are only partially contained in $K$. Since $\rho_T \rightarrow 0$ as $T \rightarrow \infty$ and the union of the measure of bins that are partially inside of $K$ can have at most a Lebesgue measure of $2\rho_T$, they will contribute negligible measure as $T$ gets large. Also, let us suppose without loss of generality that the ordering of the bins in \eqref{enum} is such that the last $n(T)$ are the ones not contained in $K$. Observing that%
\begin{align*}
&\liminf_{T \to \infty} \pi_0 \Bigl( \bigcup_{i=0}^{n_1(T)}B_i^T  \Bigr) = \liminf_{T \to \infty} \pi_0 \Bigl( K  \cap \bigcup_{i=0}^{n_1(T)}B_i^T \Bigr) = \liminf_{T \to \infty} \pi_0 \Bigl( \bigcup_{i=0}^{n_1(T) - n(T)}B_i^T \Bigr) \\
&\leq \liminf_{T \to \infty} \rho_{\max} \cdot m\Bigl( \bigcup_{i=0}^{n_1(T) - n(T)}B_i^T \Bigr) \leq  \liminf_{T \to \infty} \Bigl( \frac{\rho_{\max} \cdot 2b \cdot (n_1(T) - n(T))}{\prod_{k,l}c_{k,l}^{m_{k,l}(1-3r)T}} \Bigr),%
\end{align*}
we obtain%
\begin{equation*}
  \frac{1 - \alpha}{2b \cdot \rho_{\max}} \leq \liminf_{T \to \infty}\Bigg( \frac{(n_1(T) - n(T))}{\prod_{k,l}c_{k,l}^{m_{k,l}{(1-3r)T}}} \Bigg),%
\end{equation*}
from which we conclude that the number of bins $n_1(T) - n(T)$ which are entirely contained in $K$ must grow at an exponential rate of at least $\sum_{k,l}m_{k,l}(1-3r) \log c_{k,l}$ with $T$, just as $n_1(T)$ does. Thus, since we are concerned only with the number of bins entirely contained in $K$, we may as well assume that all are entirely in $K$ (or alternatively, relabel $n_1(T) - n(T)$ to be $n_1(T)$).%

We continue by extracting a sub-collection of disjoint bins $(C_i^T)_{i=1}^{n_2(T)}$ as described in \cite[App.~A]{kawan2019invariance}. This new sub-collection has the property that%
\begin{equation*}
  \frac{1}{2} m\bigg(\bigcup_{i=1}^{n_1(T)}B_i^T \bigg) \leq m\bigg(\bigcup_{i=1}^{n_2(T)}C_i^T \bigg).%
\end{equation*}
Also, it is clear that for any given $T$, $\frac{1}{2} n_1(T) \leq n_2(T)$. Hence, we also have the exponential growth condition of%
\begin{equation*}
	\limsup_{T \to \infty}\frac{1}{T}\log n_2(T) \geq (1-3r)\sum_{k,l}m_{k,l}\log c_{k,l}.%
\end{equation*}
Analogously to \cite{kawan2019invariance}, define the collection $(D_i^T)_{i = 1}^{n_2(T)}$\footnote{These sets should not be confused with the set $D_1,\ldots,D_m \subset \W$.} and observe that $m(D_i^T \backslash C_i^T) \leq \rho_T$ for all $i$. Finally, for a fixed $L \in \N$ we join $L$ successive $D_i^T$ blocks (see \cite[p.~27]{kawan2019invariance} for an exact formulation) to get a collection $(E_i^T)_{i = 1}^{n_3(T)}$, where $n_3(T) = \lfloor \frac{n_2(T)}{L} \rfloor + 1$, possibly adding some empty sets in the last block. Again, the following holds:%
\begin{equation*}
	\limsup_{T \to \infty}\frac{1}{T}\log n_3(T) \geq (1-3r)\sum_{k=1}^{n}\sum_{l=1}^{m}m_{k,l} \log c_{k,l},\quad m(E_i^T) \geq L\rho_T.%
\end{equation*}
We also define%
\begin{equation*}
	M_T := \bigcup_{i=1}^{n_1(T)}B_i^T 	\spc \spc \spc \overline{M}_T := \bigcup_{i=1}^{n_3(T)}E_i^T \backslash (D_{iL}^T \backslash C_{iL}^T )%
\end{equation*}
and observe that $m(M_T) \leq 2n_2(T)\rho_T \leq 2n_3(T)L\rho_T$.%

\textbf{Step 2 (Auxiliary coding scheme):} We now construct a sequence of codes to transmit information over the channel. We will transmit a quantized version of the initial state random variable $x_0$. The quantization will be done using the bins constructed earlier. For a fixed $L$ and for each $T$, we will construct a code. Note that we are considering a channel with feedback, which can be used by the encoding function. For a given $T$, the encoding and decoding processes are specified as follows.%

\textbf{Encoder}: We give to the encoder the noise realization $w$ that we have conditioned on throughout, the function $f$ corresponding to the system dynamics, and the fixed causal coding and control policy. In the classical notion of a code, the encoding function is a deterministic map and given the (system) noise sequence realization, this is the case here. The transmitted codeword is determined as follows. For an initial state realization $x_0$, the first symbol of the codeword is $q_0 = \gamma_0^e(x_0)$. Now, because the channel has feedback, the encoder can determine $u_0$ by applying the decoding function of the fixed causal coding and control policy to the output of the channel resulting from the first codeword symbol $q_0$. Thus, using the fixed and known noise realization $w$, $x_1$ can be computed. Then, the second codeword symbol $q_1= \gamma_1^e(x_0,x_1,u_0)$ is computed again using the causal coding and control policy, and so on until $q_{T-1}$ is determined (note that the encoder makes use of the channel feedback from the channel, and thus we use the generalized version of the strong converse theorem for channel capacity to obtain a contradiction). We are essentially viewing the coding and control policy as a scheme from which the initial state can be estimated at the controller end of the channel.%

\textbf{Decoder}: At time $T$, the decoder has received $T$ symbols from the channel, which are used to compute the control decisions $u_0,\ldots,u_{T-1}$ according to the fixed causal coding and control policy. The decoder also has knowledge of the noise sequence $w$ and uses it to compute the point $x_0(T,u,w)$. Our goal is to use the received channel output and control sequence to reconstruct the index $Y$ of the bin $E_Y^T$ containing $x_0$. We do this by looking at the point $x_0(T,u,w)$ for the observed control sequence $u$. Note that $w$ can be thought of as deterministic since we are conditioning on its occurrence. Recall also that $x_0(T,u,w)$ is the ``midpoint'' of the set $A_T(u,w)$, and \textit{can be computed without knowledge of the initial state $x_0$}. We simply decide on our guess $\tilde{Y}$ of the index as follows.%
\begin{itemize}
\item If $x_0(T,u,w) \in M_T$, take the index $i$ of the set $E_i^T$ containing $x_0(T,u,w)$. %
\item If $x_0(T,u,w) \notin M_T$, then decide randomly between $i$ and $i+1$, where $i$ is the index of the set $E_i^T$ that $x_0(T,u,w)$ belongs to.%
\end{itemize}

\textbf{Analysis of probability of the error for the code}. To study the probability of error, let $Y$ be a random variable on the indices $\{1,\ldots,n_3(T)\}$, where $P(Y = i) = \pi_0(E_i^T)$. We analyze $P(\tilde{Y} \neq Y)$.%

First, by construction of the bins and the estimation scheme, we have%
\begin{equation*}
  P\Bigl(\tilde{Y} \neq Y \big| x_0 \in \overline{M}_T,|x_0 - x_0(T,u,w)| \leq \frac{b}{\prod_{k,l}c_{k,l}^{m_{k,l}(1-3r)T}}\Bigr) = 0%
\end{equation*}
and%
\begin{equation*}
  P\Bigl(\tilde{Y} \neq Y \big| x_0 \in M_T \backslash \overline{M}_T,|x_0 - x_0(T,u,w)| \leq \frac{b}{\prod_{k,l}c_{k,l}^{m_{k,l}(1-3r)T}}\Bigr) \leq \frac{1}{2}.%
\end{equation*}
As such, from \eqref{conditionedbound}, it is not hard to see that for every $T$ sufficiently large,%
\begin{equation*}
  P(Y \neq \tilde{Y}) \leq \frac{1}{2}\pi_0(M_T\backslash \overline{M}_T) + \alpha.%
\end{equation*}
By an analysis exactly as in \cite{kawan2019invariance}, we have%
\begin{equation*}
  \pi_0(\M_T \backslash \overline{M}_T) \leq \frac{1}{L}\frac{\rho_{\max}}{\rho_{\min}}\pi_0(M_T).%
\end{equation*}
Combining the above two inequalities, we obtain%
\begin{equation*}
  \sum_{i=1}^{n_3(T)}P(Y = i)P(\tilde{Y} \neq Y | Y = i) \leq \frac{1}{2L}\frac{\rho_{\max}}{\rho_{\min}}\pi_0(M_T) + \alpha.%
\end{equation*}

\textbf{Step 3 (Introduction of an auxiliary uniform random variable):} In order to obtain a contradiction to the strong converse theorem for DMCs, we need to transmit a random variable \textit{uniformly} distributed on the indices $1,\ldots,n_3(T)$. Let us call this random variable $W = W_T$. Of course, at any time step, $W$ must be conditionally independent from the channel output, given the channel input. To obtain the desired contradiction, we must show that $\lim_{T \to  \infty}P(W \neq \tilde{Y}) < 1$. Before considering this quantity, note that by following exactly the same steps as in \cite{kawan2019invariance}, we obtain%
\begin{equation*}
  \pi_0(M_T) \leq \rho_{\max} \cdot m(M_T) \leq 2 n_3(T) \rho_{\max} \cdot L \rho_T%
\end{equation*}
and also%
\begin{equation*}
  \sum_{i=1}^{n_3(T)} \frac{1}{n_3(T)}P(\tilde{Y} \neq Y | Y = i) \leq \frac{\alpha + \frac{\rho_{\max}\pi_0(M_T)}{4L\rho_{\min}}}{\frac{\rho_{\min}\pi_0(M_T)}{2\rho_{\max}}}.%
\end{equation*}
Again as in \cite{kawan2019invariance} we have%
\begin{equation} \label{step4bound}
P(W \neq \tilde{Y}) = \sum_{i=1}^{n_3(T)}P(W=i)P(\tilde{Y} \neq W | W = i) \leq P(Y \neq W) + \frac{\alpha + \frac{\rho_{\max}\pi_0(M_T)}{4L\rho_{\min}}}{\frac{\rho_{\min}\pi_0(M_T)}{2\rho_{\max}}}.%
\end{equation}

\textbf{Step 4 (Application of optimal transport):} Recall the independence condition mentioned above that $W$ must satisfy. To achieve this, one could adjoin $W$ to the common probability space using the product measure, thus keeping $W$ independent from all other random variables. Observe however, that the random variable $x_0$ satisfies the independence condition that we require $W$ to satisfy. As such, we are free to choose any possible coupling between $W_T$ and $x_0$ while still ensuring that $W$ will remain independent form the channel output given the channel input (in particular, $x_0$ and $W$ need not be independent). We will take advantage of this observation.%

Consider \eqref{step4bound} and note that if the limit as $T \to \infty$ of the right-hand side is strictly less than $1$, then we will have the desired contradiction with the strong converse. As such, we proceed by finding a coupling between $W$ and $x_0$ which makes $P(Y \neq W)$ small enough so that the limit is less than $1$.%

We continue by letting $\mu$ denote the law of $Y$. That is, for every index $i \in {1,\ldots,n_3(T)}$, $\mu(i) = \pi_0(E_i^T)$. Let also $\nu$ represent the law of $W$, i.e., a uniform measure on the set $\{1,\ldots,n_3(T)\}$. We now invoke \cref{tvlemma}, which guarantees the existence of a coupling $(Y,W):(\Omega,\sF,P) \to \{1,\ldots,n_3(T)\}^2$ such that%
\begin{equation*}
  P(Y \neq W) = \frac{1}{2}\sum_{i=1}^{n_3(T)}|\mu(i)-\nu(i)|.%
\end{equation*}
Let now $A = \{ i \in \{1,\ldots,n_3(T)\} :\mu(i)\geq \nu(i)  \}$ and observe that%
\begin{align*}
& 1 - \sum_{i=1}^{n_3(T)}\min(\mu(i),\nu(i)) = \frac{1}{2}\sum_{i=1}^{n_3(T)}\mu(i) + \frac{1}{2}\sum_{i=1}^{n_3(T)}\nu(i) - \sum_{i \in A} \nu(i) - \sum_{i \in A^c} \mu(i)	\\
&=\frac{1}{2}\sum_{i \in A}\mu(i) - \frac{1}{2}\sum_{i \in A^c}\mu(i) - \frac{1}{2}\sum_{i \in A}\nu(i) + \frac{1}{2}\sum_{i \in A^c}\nu(i) 	\\
&=\frac{1}{2} \Big( \sum_{i \in A}\mu(i) - \nu(i) \Big)   + \frac{1}{2} \Big( \sum_{i \in A^c}\nu(i) - \mu(i)\Big) = \frac{1}{2}\sum_{i=1}^{n_3(T)}|\mu(i)-\nu(i)|.%
\end{align*}
Thus, we can write%
\begin{equation*}
  P(Y \neq W) = \frac{1}{2}\sum_{i=1}^{n_3(T)}|\mu(i)-\nu(i)| = 1 - \sum_{i=1}^{n_3(T)}\min(\mu(i),\nu(i)).%
\end{equation*}
To get an upper bound for the right-hand side, note that%
\begin{align*}
\mu(i) &= \pi_0(E_i^T) \geq \rho_{\min} \cdot m(E_i^T) = \frac{n_3(T)}{n_3(T)} \cdot m(E_i^T) \cdot \rho_{\min} \\
&\geq \frac{n_2(T) \cdot  \rho_{T} \cdot \rho_{\min}  }{n_3(T)} \geq \frac{ m(M_T) \cdot \rho_{\min}}{2 \cdot n_3(T)} \geq \frac{ \pi_0(M_T) \cdot \rho_{\min}}{2 \cdot \rho_{\max} \cdot n_3(T)} \geq \frac{\rho_{\min} \cdot (1-\alpha) }{2 \cdot \rho_{\max} \cdot n_3(T)}.%
\end{align*}
Recalling that $\nu(i) = 1/n_3(T)$ for each $i$, we have $\min(\mu(i),\nu(i)) \geq (\rho_{\min} \cdot (1-\alpha))/(2 \cdot \rho_{\max} \cdot n_3(T))$ for all $i$, and therefore%
\begin{equation*}
  P(Y \neq W) \leq 1 - \frac{\rho_{\min} \cdot (1-\alpha) }{2 \cdot \rho_{\max} }.%
\end{equation*}
Combining with \eqref{step4bound}, we obtain%
\begin{equation*}
  P(W \neq \tilde{Y}) \leq 1 - \frac{\rho_{\min} \cdot (1-\alpha) }{2 \cdot \rho_{\max}} + \frac{\alpha + \frac{\rho_{\max}\pi_0(M_T)}{4L\rho_{\min}}}{\frac{\rho_{\min}\pi_0(M_T)}{2\rho_{\max}}}%
\end{equation*}
which holds for all $T$ sufficiently large. We now evaluate the right-hand side to determine its behavior as $T$ tends to infinity. We have%
\begin{align*}
 & \limsup_{T \to \infty} \Big(  1 - \frac{\rho_{\min} \cdot (1-\alpha) }{2 \cdot \rho_{\max}} + \frac{\alpha + \frac{\rho_{\max}\pi_0(M_T)}{4L\rho_{\min}}}{\frac{\rho_{\min}\pi_0(M_T)}{2\rho_{\max}}} \Big) \\
&\leq 1 - \frac{\rho_{\min} \cdot (1-\alpha) }{2 \cdot \rho_{\max}} + \frac{\rho_{\max}^2}{2L\rho_{\min}^2} + \frac{2 \cdot \alpha \cdot \rho_{\max}}{\rho_{\min}}\limsup_{T \to \infty}\frac{1}{\pi_0(M_T)} \\
&\leq 1 - \frac{\rho_{\min} \cdot (1-\alpha) }{2 \cdot \rho_{\max}} + \frac{\rho_{\max}^2}{2L\rho_{\min}^2} + \frac{2 \cdot \rho_{\max}}{\rho_{\min}} \frac{\alpha}{1-\alpha},%
\end{align*}
where the last inequality follows from \eqref{eq_bit_measure}. Recall now that throughout, $L \in \N$ was fixed but arbitrary. Taking $L$ large enough so that \eqref{alphaassumption} holds, and writing $T$-subscripts to emphasize $T$-dependence, we obtain $\limsup_{T \to \infty}P(W_T \neq \tilde{Y}_T) < 1$, which is a contradiction, since it negates the strong converse theorem for DMCs with feedback. Hence, the proof is complete.
\end{proof}

\section{Appendix}\label{sec:app}

In this section, we state a few results required in the paper.

\subsection{A result from optimal transport}

In the proof of \cref{thm;NewTheorem2}, a basic result from optimal transport is used, which we state here.%

\begin{definition}
Let $\mu$ and $\nu$ be Borel probability measures on a metric space $(S,d)$. A coupling of $\mu$ and $\nu$ is a pair of random variables $X,Y$ defined on some probability space $(\Omega,\sF,P)$ such that the law of the random variable $(X,Y)$ on $S^2$ admits $\mu$ and $\nu$ as its marginals.
\end{definition}

The notion of coupling can easily be generalized for the case where the measures $\mu$ and $\nu$ are on distinct spaces, however we do not require that level of generality. The total variation distance between probability measures on the same measurable space serves as a measure for how distinct they are. The definition reads as follows.%

\begin{definition}
Let $\mu$ and $\nu$ be probability measures on a measurable space $(\Omega,\sF)$. We define the total variation distance as%
\begin{equation*}
  \norm{\mu - \nu}_{TV} := 2 \sup_{A \in \sF}|\mu(A)-\nu(A)|.%
\end{equation*}
\end{definition}%
\begin{lemma}\label{tvlemma}
Let $(X,Y):(\Omega,\sF,P) \rightarrow S^2$ be a coupling of the probability measures $\mu$ and $\nu$ on the metric space $(S,d)$. Then%
\begin{equation*}
  \norm{\mu - \nu}_{\mathrm{TV}} \leq 2 \cdot P(\{\omega \in \Omega : X(\omega) \neq Y(\omega)\}).%
\end{equation*}
\end{lemma}
If in addition, $S$ is a finite set, then a coupling $(X,Y)$ exists which achieves the above bound. 
\begin{proof}
See Equation (6.11) in \cite{villani2008optimal}.
\end{proof}
Note also that if $S$ is finite in the above setup, then a simple calculation results in%
\begin{equation*}
  \norm{\mu - \nu}_{\mathrm{TV}} = \sum_{x \in S}|\nu(x) - \mu(x)|.%
\end{equation*}
Indeed, for finite $S$ let $A \coloneqq \{x \in S : \mu(x) \geq \nu(x)\}$. The result follows by noting that $\norm{\mu - \nu}_{\mathrm{TV}} = |\mu(A) - \nu(A)| + |\mu(A^c) - \nu(A^c)|$. As such, a coupling $(X,Y)$ of the laws exists which satisfies%
\begin{equation*}
  P(X \neq Y) = \frac{1}{2} \sum_{x \in S}|\nu(x) - \mu(x)|.%
\end{equation*}
We make use of this identity in case 2 of the proof for the noisy channel case.%

\subsection{Channel coding theorem}\label{infoBackground}

When considering a system controlled over a noisy channel, we make use of the strong converse of the noisy channel coding theorem. We state the necessary definitions and theorems here without proof. A detailed overview of these concepts can be found in \cite{cover1999elements}.

\begin{definition}
Consider a memoryless finite alphabet channel with input alphabet $\sX$, output alphabet $\sY$ and a given transition probability measure. The capacity of the channel is defined by $C \coloneqq \sup_{p(x)}I(\sX,\sY)$, where the sup is taken over all possible probability measures on the input alphabet $\sX$. We call such a channel a Discrete Memoryless Channel (DMC). A DMC with feedback is as above, but with the additional property that the encoder has knowledge of the channel output. It is well-known that feedback does not increase channel capacity.
\end{definition}

Next, we provide the definition of a code. We provide the definitions for channels without feedback, however the feedback case is very similar, the only difference being that at a given time, the encoder can use the channel output for previous inputs in generating the next codeword symbol.%

\begin{definition}
For $M,n\in\N$, an $(M,n)$-code consists of an encoding function $x^n:\{1,\ldots,M\} \to \sX^n$ and a decoding function $g: \sY^n \to \{1,\ldots,M\}$. We define the rate of an $(M,n)$-code by $R \coloneqq (\log M)/n$.
\end{definition}

For a code as above, we call $x^n(1),x^n(2),\ldots,x^n(M)$ the codewords. Because the channel distorts the codewords, we must consider the probability that we can decode correctly. This leads to the following definition.%

\begin{definition}
The maximal error of an $(M,n)$-code is given by%
\begin{equation*}
  \lambda^{(n)} := \max_{i = 1,\ldots,M}P(g(Y^n) \neq i|X^n = x^n(i)).%
\end{equation*}
\end{definition}

\begin{definition}
A rate $R$ is called achievable if there exists a sequence of $(\Ceil{2^{nR}},n)$-codes with the property that $\lambda^{(n)} \to 0$ as $n \to \infty$.
\end{definition}

The following is the strong converse of the noisy channel coding theorem in information theory.%

\begin{theorem}\label{thm;strongconv}
Consider a DMC $(\sX,p(\cdot|\cdot),\sY)$ of capacity $C$. Let $R>C$ and consider an arbitrary sequence of $(\Ceil{2^{nR}},n)$-codes, used to transmit the uniform random variables $W_n$, uniformly distributed on the set $\{1,\ldots,2^{nR}\}$, respectively. Then $P(W_n \neq g_n(Y^n)) \to 1$ as $n \to \infty$.
\end{theorem}

The above theorem also holds for DMCs with feedback (see \cite{kemperman1971strong} for a proof). In the proof of \cref{thm;NewTheorem2}, the encoding functions require that the channel has feedback, hence the need for this assumption in the theorem statement.%

\bibliographystyle{siamplain}
\bibliography{references}

\end{document}